\documentclass[11pt]{article}
\usepackage{amsmath,amssymb,amsthm}
\usepackage{graphicx}
\usepackage{epsfig}

\newtheorem{theo}{Theorem}[section]

\newtheorem{cor}[theo]{Corollary}

\newtheorem{defi}[theo]{Definition}

\newcounter{listagem}
\newcommand{\blista}{\begin{list}{\roman{listagem})}{\usecounter{listagem}}}
\newcommand{\elista}{\end{list}}
\newcommand{\BR}{{\mathbb R}}
\newcommand{\BN}{{\mathbb N}}
\newcommand{\BZ}{{\mathbb Z}}
\newcommand{\BC}{{\mathbb C}}
\newcommand{\BH}{{\mathbb H}}

\newcommand{\e}{{\bf e}}

\newcommand{\cW}{{W}}

\newcommand{\beq}{\begin{equation}}
\newcommand{\eeq}{\end{equation}}
\newcommand{\beqn}{\begin{eqnarray}}
\newcommand{\eeqn}{\end{eqnarray}}
\newcommand{\ov}{\overline}
\newcommand{\ul}{\underline}

\newcommand{\Cl}{{C \kern -0.1em \ell}}
\newcommand{\R}{\mathbb{R}}

\newcommand{\f}{\mathfrak{f}}
\newcommand{\fm}{\mathfrak{f}^{\dagger}}

\begin{document}

\title{Numerical Clifford Analysis for Nonlinear Schr\"odinger Problem}

\pagestyle{myheadings} \markboth{P. Cerejeiras, N. Faustino and N.
Vieira}{Numerical Clifford Analysis for Nonlinear Schr\"odinger
Problem}

\author{P. Cerejeiras$^*$ \hspace*{1cm} N. Faustino\thanks{Departamento de Matem\'atica, Universidade de
Aveiro, Portugal}\hspace*{1cm} N. Vieira$^*$ \\
{\small pceres@mat.ua.pt \hspace*{1cm} nfaust@mat.ua.pt
\hspace*{1cm} nvieira@mat.ua.pt}} \maketitle

\begin{abstract}
The aim of this work is to study the numerical solution of the
nonlinear Schr\"odinger problem using a combination between Witt
basis and finite difference approximations. We construct a discrete
fundamental solution for the non-stationary Schr\"odinger operator
and we show the convergence of the numerical scheme. Numerical
examples are given at the end of the paper
\end{abstract}



\section{Introduction}

In this paper we make use of Clifford analysis tools in order to
treat a well-known partial differential equation of mathematical
physics. This treatment is based on the work developed by K.
G\"urlebeck and W. Spr\"o\ss ig in ~\cite{GS} and it is (partially)
based on an orthogonal decomposition of the underlying function
space in terms of the subspace of null-solutions of the
corresponding Dirac operator. While the orthogonal decomposition of
G\"urlebeck and Spr\"o\ss ig  has been applied with success to PDE's
such as Lam\'e equations, Maxwell equations and Navier-Stokes
equations \cite{KK}, it works for the stationary case only.

In \cite{CKS} an alternative approach was proposed, based on an
adding of extra basis elements, namely, of a Witt basis. This
approach allows the application of the already existent techniques
of elliptic function theory developed in \cite{GS} to time-varying
domains. A suitable orthogonal decomposition for the underlying
function space is, therefore, obtained in terms of the kernel of the
positive parabolic Dirac operator and its range after application to
a Sobolev space with zero boundary values.

After some basic notions about Clifford algebras presented in the
next section, we will define, in Section 2, a generalization of the
parabolic Dirac operator introduced in \cite{CKS} and a
generalization of the Teodorescu and Cauchy-Bitsadze operators
presented in \cite{GS}. Moreover, using the previous definitions we
will obtain a factorization of our equation in terms of basis
elements of Witt basis and we obtain the fundamental solution for
our generic parabolic Dirac operator.

However, the integral representation formulae obtained via this
theoretical method are not suitable for an explicit computation of
the solution, due to unacceptable convergence rates of the
integrals' numerical approximation (see \cite{GS} for more details).
Hence, to avoid this backdraw it becomes necessary to study the
discrete analogues of the operators, namely discrete counterparts
for the single- and double-layer potentials. Contrary to difference
potentials introduced by Ryabenkij \cite{Ryaij}, where the
difference potentials are constructed by means of discrete Green
functions, we will introduce in Section 3 difference potentials
based on the discrete fundamental solution. An advantage of this
approach is that, contrary to discrete Green functions, we will
obtain an explicit expression for our discrete fundamental solution
$E_{h,-i\tau}$ which is independent of the choice or shape of the
domain. In Section 4 we prove the convergence of the discrete
counterparts of the analytic operators introduced in Section 2. This
will allow us to establish a convergent numerical scheme for the
linear non-stationary Schr\"odinger equation.

In Section 5 we will adapt the previous algorithm in order to solve
numerically the cubic Schr\"odinger equation and we will present in
Section 6 some simple numerical examples to show the consistency and
stability of our algorithm for different mesh sizes $h$ and $\tau.$


\section{Preliminaries}


\subsection{Clifford algebras}

Consider the $n$-dimensional vector space $\R^n$ endowed with a
standard orthonormal basis $\{ e_1, \cdots, e_n \}$  and satisfying
the multiplication rules $e_i e_j + e_j e_i = -2\delta_{i,j}.$

We define the universal Clifford algebra $\Cl_{0,n}$ as the
$2^{n}$-dimensional associative algebra with basis given by $ e_{0}
= 1$ and $ e_A = e_{h_1} \cdots e_{h_k},$ where $A = \{ h_1, \ldots,
h_k \} \subset N = \{ 1, \ldots, n \}$, for $1 \leq h_1 < \cdots <
h_k \leq n$. Each element $x \in \Cl_{0,n}$ will be represented by
$x=\sum_{A} x_A e_A$ and each non-zero vector $x = \sum_{j=1}^n x_j
e_j\in \BR^{n}$ has a multiplicative inverse given by
$\frac{-x}{|x|^2}$. We denote by $\ov{x}^{\Cl_{0,n}}$ the (Clifford)
conjugate of the element $x \in \Cl_{0,n},$ where
    \begin{eqnarray*}
    \ov{1}^{\Cl_{0,n}}=1 \qquad \ov{e_j}^{\Cl_{0,n}}
     = -e_j \qquad \ov{ab}^{\Cl_{0,n}}=\ov{b}^{\Cl_{0,n}}\ov{a}^{\Cl_{0,n}}.
    \end{eqnarray*}

We introduce the complexified Clifford algebra $\Cl_n$ as the
tensorial pro\-duct
$$    \mathbb{C} \otimes \Cl_{0,n} = \{ w=\sum_{A} z_A e_A , ~ z_A \in
    \mathbb{C}, A \subset N \},$$ where the imaginary unit interact
    with the basis elements as $i e_j = e_j i,~j= 1, \ldots, n.$

The conjugation is defined as $\ov{w} = \sum_{A}
\ov{z_A}^{\mathbb{C}} \ov{e_A}^{\Cl_{0,n}}.$\\

We consider the Dirac operator $D=\sum_{j=1}^{n} e_j\frac{\partial}
{\partial x_i}$ which has the property of factorizing the
$n$-dimensional Laplacian, that is, $D^2=-\Delta$. A $\Cl_n$-valued
function on an open domain $\underline \Omega,$ $u:\underline \Omega
\subset \BR^n \mapsto\Cl_n,$ is said to be {\it left-monogenic} if
it satisfies $Du=0$ on $\underline \Omega.$

Let now $\Omega \subset \BR^{n}\times \BR^+$ denote a bounded domain
with a sufficiently smooth boundary $\Gamma=\partial\Omega,$ while
$(0,T),$ with $T>0$, represents its projection on the time-domain. A
function $u:\Omega \mapsto \Cl_{n}$ has a representation $u=\sum_A
u_A e_A$ with $\BC $-valued components $u_A$. Properties such as
continuity will be understood component-wisely. In the following we
will use the short notation $L_p(\Omega)$, $C^k(\Omega)$, etc.,
instead of $L_p(\Omega,\Cl_n)$, $C^k(\Omega,\Cl_n)$. For more
details, see \cite{DSS}.

Taking into account \cite{CKS} we will imbed $\mathbb{R}^n$ into
$\mathbb{R}^{n+2}$. For that purpose we add two new basis elements
$\f$ and $\fm$ satisfying
    \begin{eqnarray}
         {\f}^2 = {\fm}^2 = 0, &  \f \fm + \fm \f = 1, &
         \f e_j +e_j \f =  \fm e_j +e_j \fm = 0, j=1,\cdots,n.
         \label{regras}
    \end{eqnarray}
The set $\{ \f, \, \fm\}$ is said to be a Witt basis for
$\mathbb{R}^2$ and it will allows us to create a suitable
factorization of the Schr\"odinger operator where only partial
derivatives are used.


\subsection{Factorization of time-evolution operators}

In this section we present a new method for factorizing the
Schr\"odinger equation,
    \begin{equation}
   (\pm i \partial_t -\Delta ) u(x,t)=0, ~~(x,t) \in \Omega,
   \label{Eq:1}
    \end{equation}
where $\Omega \subset \BR^n \times \BR^+$ denotes a bounded domain.
For this we will follow the ideas presented in \cite{CKS}, \cite{B}
and \cite{CV}.

\begin{defi}
For a function $u \in C^1(\Omega)$ we define the forward (resp.
backward) parabolic Dirac operator
    \begin{eqnarray}
    D_{x,\pm it}u= (D + \f \partial_t \pm i \fm)u, \label{I}
    \end{eqnarray}
where $D$ stands for the (spatial) Dirac operator.
\end{defi}
These operators factorize the correspondent backward/forward
time-evolution operator (\ref{Eq:1}), that is
    \begin{eqnarray}
    (D_{x, \pm it})^2 u = (\pm i \partial_t-\Delta) u. \label{II}
    \end{eqnarray}

We consider now the generic Stokes' Theorem.
\begin{theo}\label{Th:22}
For each $u, v \in W^1_p(\Omega),$ $1<p<\infty$ it holds
    \begin{eqnarray*}
    \int_{\partial \Omega} v d \sigma_{x,t} u & = & \int_{\Omega}  [(v D_{x, -it}) u + v ( D_{x, +it} u )] dx dt,
    \end{eqnarray*}
where  $d \sigma _{x,t} = (D_x+\f \partial_t) \rfloor dxdt$ stands
for the contraction of the homogeneous operator associated to $D_{x,
-it}$ with the volume element.
\end{theo}

For the proof of this theorem we refer to \cite{CKS}.

We shall construct a fundamental solution for the backward parabolic
Dirac operator $D_{x,-it}$ in terms of a fundamental solution of the
backward Schr\"o\-din\-ger operator. We recall that the function
\begin{eqnarray}
    e_-(x,t) & = &  i \frac{H(t)}{(4\pi i t)^{n/2}} \exp \left(i\frac{|x|^2}{4t}
    \right) \label{Form:2}
    \end{eqnarray}
is a fundamental solution for the backward Schr\"odinger operator
since it satisfies
$$ (-i\partial_t - \Delta)e_-(x,t) =     e_-(x,t)(-i\partial_t-\Delta)  = \delta(x, t)$$
in distributional sense. Therefore, we have
\begin{defi}
    Given a fundamental solution $e_-=e_-(x,t)$ for the backward
    Schr\"odinger operator we have as a fundamental solution $E_-=E_-(x,t)$ for
    the backward parabolic Dirac operator $D_{x,-it}$ the function
 \begin{eqnarray}
    E_-(x,t) & = &  e_-(x,t) D_{x,-it} \nonumber \\
    & = & \frac{H(t)\exp \left(i \frac{|x|^2}{4t} \right) }{(4 \pi i t)^{n/2}}
    \left( -\frac{x}{2t} + \f \left( \frac{ |x|^2}{4t^2}  -i\frac{n}{2t} \right) + \fm
    \right). \label{Form:1}
    \end{eqnarray}
\end{defi}

Using the fundamental solution (\ref{Form:1}) and the generic
Borel-Pompeiu formula we construct the adequate Teodorescu and
Cauchy-Bitsadze operators.
    \begin{defi}\label{Def:24}
    For a function $u \in L_p(\Omega), ~1<p<\infty,$ we define the correspondent Teodorescu and Cauchy-Bitsadze
    operators, respectively,  as
    \begin{eqnarray*}
    Tu(x_0,t_0) & = & \int_{\Omega} E_{-}(x-x_0,t-t_0)u(x,t)dx dt,  \\
    Fu(x_0,t_0) & = & \int_{\partial \Omega} E_{-}(x-x_0,t-t_0)d\sigma_{x,t} u(x,t).
    \end{eqnarray*}
    \end{defi}

We also have the following decomposition (c.f.~\cite{CV}).

\begin{theo}
The space $L_p(\Omega)$, $1<p \leq 2$ allows the direct
decomposition
$$    L_p(\Omega) =  L_p(\Omega) \cap \textrm{ker}\left( D_{x,+it}\right) \oplus D_{x,+it}
    \left( \stackrel{\circ}{W_p^{1}} \left(\Omega
    \right)\right),$$where $\stackrel{\circ}{W_p^{1}} \left(\Omega
    \right)$ denotes the space of all functions in the Sobolev space $W_p^{1} \left(\Omega
    \right)$ with zero-boundary values.
\end{theo}

The previous decomposition of the $L_p$-space allows us to establish
two projections operators.

\begin{defi}\label{def:21} Let $1<p \leq 2.$ We define the projectors
\begin{equation}
P : L_p(\Omega) \rightarrow L_p(\Omega) \cap \textrm{ker}\left(
D_{x,+it}\right) \label{cont_P}
\end{equation}
and
\begin{equation}
Q : L_p(\Omega) \rightarrow D_{x,+it} \left(
\stackrel{\circ}{W_p^{1}} \left(\Omega     \right)\right).
\label{cont_Q}
\end{equation}
\end{defi}

\begin{theo} \label{Th:26}Let $f \in L_p(\Omega),$ for $ 1<p \leq 2.$ The solution of the forward linear Schr\"odinger problem
    $$\left \{
    \begin{array}{rcl}
 i \frac{\partial u}{\partial t} - \Delta u & = & f ~ in ~ \Omega  \\
    u & = & 0 ~on ~ \partial \Omega
    \end{array}
    \right. $$
is then given by $u=TQT f.$
\end{theo}

The proof of this theorem was made in \cite{CV} for the case of
$p=2$. However, we remark that it can easily be extended to $1< p
<2$. Moreover,

1) we can obtain dual results for the backward Schr\"odinger problem
by considering a fundamental solution for the forward parabolic
Dirac operator $D_{x,+it}$ on Theorem \ref{Th:22};

2) the above construction can easily be generalized for arbitrary
operators of type $a\partial_t -\Delta, $ where $a$ is a non-zero
complex parameter. Indeed, the case of $a=1$ gives the well-known
heat equation while for $a=i$ we have the non-stationary
Schr\"odinger equation.


\section{Discrete fundamental solution for the time-evo\-lu\-tion problem}


\subsection{Quaternionic matrix representation of the Witt Basis}

We use the matrix representation of the generators of the real
quaternions as defined in \cite{GS},
    \begin{eqnarray*}
    \e_0=\left(
    \begin{array}{cccc}
      1 & 0 & 0 & 0 \\
      0 & 1 & 0 & 0 \\
      0 & 0 & 1 & 0 \\
      0 & 0 & 0 & 1 \\
    \end{array}\right),  &
    \e_1=\left(
    \begin{array}{cccc}
      0 & 1 & 0 & 0 \\
      -1 & 0 & 0 & 0 \\
      0 & 0 & 0 & -1 \\
      0 & 0 & 1 & 0 \\
    \end{array}\right),
    \end{eqnarray*}
    \begin{eqnarray*}
    \e_2=\left(
    \begin{array}{cccc}
      0 & 0 & -1 & 0 \\
      0 & 0 & 0 & -1 \\
      1 & 0 & 0 & 0 \\
      0 & 1 & 0 & 0 \\
    \end{array}\right), &
    \e_3=\left(
    \begin{array}{cccc}
      0 & 0 & 0 & -1 \\
      0 & 0 & 1 & 0 \\
      0 & -1 & 0 & 0 \\
      1 & 0 & 0 & 0 \\
    \end{array}\right),
    \end{eqnarray*} as representatives of a discrete version of the spatial basis for the
quaternionic case.


\subsection{Finite differences and time evolution operators}

As already stated we want to investigate a finite difference scheme
based on the notion of a discrete fundamental solution as described in
\cite{GH}. We denote by
    $$\BR^3_h = \{h\ul{m}=(hm_1,hm_2,hm_3),m_l\in\BZ\} \mbox{ ~~ and ~~}
    \BR^+_\tau = \{k\tau, k \in \BZ^+\}$$
equidistant lattices corresponding to space and time discretization,
respectively. For a discrete function $u : \BR^3_h \times \BR^+_\tau
\rightarrow \BC^4 \sim \BC \otimes \BH,$   $u(h \ul m, k \tau )=
(u^0, u^1, u^2, u^3),$ we have the finite difference approximation
for the stationary Dirac operators given by {\small
\begin{eqnarray*}
    \begin{array}{ccc}
    D_h^{-+}u = \left(
    \begin{array}{c}
    -\partial_h^{-1}u^1-\partial_h^{-2}u^2-\partial_h^{-3}u^3\\
    \partial_h^{-1}u^0-\partial_h^{3}u^2+\partial_h^{2}u^3\\
    \partial_h^{-2}u^0+\partial_h^{3}u^1-\partial_h^{1}u^3\\
    \partial_h^{-3}u^0-\partial_h^{2}u^1+\partial_h^{1}u^2
    \end{array}
    \right),
    \end{array}
    \begin{array}{ccc} &
    D_h^{+-}u  =  \left(
    \begin{array}{c}
    -\partial_h^{1}u^1-\partial_h^{2}u^2-\partial_h^{3}u^3\\
    \partial_h^{1}u^0-\partial_h^{-3}u^2+\partial_h^{-2}u^3\\
    \partial_h^{2}u^0+\partial_h^{-3}u^1-\partial_h^{-1}u^3\\
    \partial_h^{3}u^0-\partial_h^{-2}u^1+\partial_h^{-1}u^2
    \end{array}
    \right),
    \end{array} \\ & \\
    \begin{array}{ccc}
u D_h^{-+} = \left(
    \begin{array}{c}
     -\partial_h^{-1}u^1-\partial_h^{-2}u^2-\partial_h^{-3}u^3\\
    \partial_h^{-1}u^0+\partial_h^{3}u^2-\partial_h^{2}u^3\\
    \partial_h^{-2}u^0-\partial_h^{3}u^1+\partial_h^{1}u^3\\
    \partial_h^{-3}u^0+\partial_h^{2}u^1-\partial_h^{1}u^2
    \end{array}
    \right),
    \end{array}
    \begin{array}{ccc} &
    u D_h^{+-}  =  \left(
    \begin{array}{c}
    -\partial_h^{1}u^1-\partial_h^{2}u^2-\partial_h^{3}u^3\\
    \partial_h^{1}u^0+\partial_h^{-3}u^2-\partial_h^{-2}u^3\\
    \partial_h^{2}u^0-\partial_h^{-3}u^1+\partial_h^{-1}u^3\\
    \partial_h^{3}u^0+\partial_h^{-2}u^1-\partial_h^{-1}u^2
    \end{array}
    \right),
    \end{array}
    \end{eqnarray*}}
where
    $$\partial_h^{\pm s}u^j = \frac{(u^j(h\ul{m}\pm
    h e_s,k\tau)-u^j(h\ul{m},k\tau))}{h},~j=0,1,2,3, ~s=1,2,3,$$
represent the spatial forward/backward difference operators. We
remark that these difference Dirac operators factorize the star
discretization of the Laplace operator, in the sense that
$$    D_h^{+-}D_h^{-+} = D_h^{-+}D_h^{+-}  = - \Delta_h \e_0 = \left( \sum_{s=1}^3
    \partial_h^{-s}\partial_h^s \right) \e_0.$$

Moreover, we also have the following (forward) time difference
operator (see \cite{GS}, \cite{GH})
    \begin{eqnarray*}
    \partial_{\tau}u^j (h\ul{m},k \tau ) & = &
    \frac{u^j(h\ul{m},
    (k+1)\tau)-u^j(h\ul{m},k\tau)}{\tau},~j=0,\cdots,3.
    \end{eqnarray*}

With the previous definitions we aim to construct a finite
difference approximation for the parabolic Dirac operators. For this
purpose we introduce the matrix representations
    \begin{gather}
    D_{h,\pm i\tau} =
 {\small    \left(
    \begin{array}{cc}
      {\bf 0} & D^{-+}_h \\
      D^{+-}_h & {\bf 0} \\
    \end{array} \right) + \left(
    \begin{array}{cc}
     \partial_{\tau} \e_0 & {\bf 0} \\
       {\bf 0} & \partial_{\tau} \e_0  \\
    \end{array} \right) \gamma^+
    \pm \left(
    \begin{array}{cc}
     i \e_0 & {\bf 0} \\
       {\bf 0} & i \e_0  \\
    \end{array} \right) \gamma^-,} \label{D-+htau}
    \end{gather} where $\gamma^+, \gamma^-$ denote elements which satisfy the following matricial
    operations
\begin{gather}
\gamma^\pm \left(
    \begin{array}{cc}
     A & B \\
       C & D  \\
    \end{array} \right) = \left(
    \begin{array}{cc}
     A & -B \\
      - C & D  \\
    \end{array} \right) \gamma^\pm, \nonumber \\
     (\gamma^\pm)^2 = 0, \label{actionWitt} \\
     \gamma^+ \gamma^- + \gamma^- \gamma^+ = id. \nonumber
\end{gather}

Using the properties of the previous operators and taking account
the multiplication rules (\ref{actionWitt}) we obtain the following
relation
 {\small \begin{gather}
(D_{h,\pm i\tau} )^2 = \left[ \left(
    \begin{array}{cc}
      {\bf 0} & D^{-+}_h \\
      D^{+-}_h & {\bf 0} \\
    \end{array} \right) + \left(
    \begin{array}{cc}
     \partial_{\tau} \e_0 & {\bf 0} \\
       {\bf 0} & \partial_{\tau} \e_0  \\
    \end{array} \right) \gamma^+
    \pm \left(
    \begin{array}{cc}
     i \e_0 & {\bf 0} \\
       {\bf 0} & i \e_0  \\
    \end{array} \right) \gamma^-  \right]^2  \nonumber  \\
= \left[ \left(
    \begin{array}{cc}
      {\bf 0} & D^{-+}_h \\
      D^{+-}_h & {\bf 0} \\
    \end{array} \right) \right]^2 +  \left[ \left(
    \begin{array}{cc}
     \partial_{\tau} \e_0 & {\bf 0} \\
       {\bf 0} & \partial_{\tau} \e_0  \\
    \end{array} \right) \gamma^+  \right]^2 + \left[ \left(
    \begin{array}{cc}
     i \e_0 & {\bf 0} \\
       {\bf 0} & i \e_0  \\
    \end{array} \right) \gamma^-   \right]^2 + \nonumber \\
    \left[ \left(
    \begin{array}{cc}
      {\bf 0} & D^{-+}_h \\
      D^{+-}_h & {\bf 0} \\
    \end{array} \right)   \left(
    \begin{array}{cc}
     \partial_{\tau} \e_0 & {\bf 0} \\
       {\bf 0} & \partial_{\tau} \e_0  \\
    \end{array} \right) \gamma^+ +  \left(   \begin{array}{cc}
     \partial_{\tau} \e_0 & {\bf 0} \\
       {\bf 0} & \partial_{\tau} \e_0  +  \end{array} \right) \gamma^+  \left(
    \begin{array}{cc}
      {\bf 0} & D^{-+}_h \\
      D^{+-}_h & {\bf 0} \\
    \end{array} \right)  \right] \pm  \nonumber \\
   \left[ \left(
    \begin{array}{cc}
      {\bf 0} & D^{-+}_h \\
      D^{+-}_h & {\bf 0} \\
    \end{array} \right)   \left(
    \begin{array}{cc}
     i \e_0 & {\bf 0} \\
       {\bf 0} & i \e_0  \\
    \end{array} \right) \gamma^- +  \left(   \begin{array}{cc}
    i \e_0 & {\bf 0} \\
       {\bf 0} & i \e_0  +  \end{array} \right) \gamma^-  \left(
    \begin{array}{cc}
      {\bf 0} & D^{-+}_h \\
      D^{+-}_h & {\bf 0} \\
    \end{array} \right)  \right] \pm  \nonumber \\
   \left[ \left(
    \begin{array}{cc}
     \partial_{\tau} \e_0 & {\bf 0} \\
       {\bf 0} & \partial_{\tau} \e_0  \\
    \end{array} \right) \gamma^+   \left(
    \begin{array}{cc}
     i \e_0 & {\bf 0} \\
       {\bf 0} & i \e_0  \\
    \end{array} \right) \gamma^- +  \left(   \begin{array}{cc}
    i \e_0 & {\bf 0} \\
       {\bf 0} & i \e_0  +  \end{array} \right) \gamma^-  \left(
    \begin{array}{cc}
     \partial_{\tau} \e_0 & {\bf 0} \\
       {\bf 0} & \partial_{\tau} \e_0  \\
    \end{array} \right) \gamma^+  \right]   \nonumber \\
 = \left(
    \begin{array}{cc}
      -\Delta_h & {\bf 0} \\
      {\bf 0} & -\Delta_h \\
    \end{array} \right)  + \left[ \left(
    \begin{array}{cc}
      {\bf 0} & \partial_{\tau} D^{-+}_h \\
      \partial_{\tau} D^{+-}_h & {\bf 0} \\
    \end{array} \right)   +   \left(
    \begin{array}{cc}
      {\bf 0} & -\partial_{\tau} D^{-+}_h \\
      -\partial_{\tau} D^{+-}_h & {\bf 0} \\
    \end{array} \right)  \right] \gamma^+ \pm   \nonumber \\
 \left[ \left(
    \begin{array}{cc}
      {\bf 0} & iD^{-+}_h \\
      iD^{+-}_h & {\bf 0} \\
    \end{array} \right)   +  \left(
    \begin{array}{cc}
      {\bf 0} & -i D^{-+}_h \\
      -i D^{+-}_h & {\bf 0} \\
    \end{array} \right)  \right] \gamma^-  \pm  \left(
    \begin{array}{cc}
     i \partial_{\tau} \e_0 & {\bf 0} \\
       {\bf 0} & i \partial_{\tau} \e_0  \\
    \end{array} \right) (\gamma^+    \gamma^- +    \gamma^-
    \gamma^+)   \nonumber  \\
= (-\Delta_h \pm i
\partial_{\tau}) \left(
    \begin{array}{cc}
 \e_0 & {\bf 0} \\
       {\bf 0} &  \e_0  \\
    \end{array} \right) \label{factH}
\end{gather} }
i.e., these operators factorize the difference discretization of our
time evolution operator (\ref{Eq:1}). Moreover, due to the fact that
the above finite difference operators $D^{-+}_h$, $D_h^{+-}$ and
$\partial_\tau$ are approximations of the Dirac operator $D$ and of
the time partial derivative operator $\partial_t,$ respectively (see
\cite{H}), we have that (\ref{D-+htau}) are a finite difference
approximations for the parabolic Dirac operators $D_{x,\pm it}.$


\subsection{Discrete fundamental solutions}
Based on the ideas presented in \cite{H} we introduce the discrete
fundamental solution for the Schr\"odinger difference operator $-i
\partial_{\tau}-\Delta_h$ as
 \begin{eqnarray}
    e_{h,-i\tau} (h\ul{m},k \tau) & = &  i H(k\tau ) \left (  1 + i \tau \Delta_h \right)^{k-1}
    \delta_{h} (h\ul{m}), \label{fundLapl}
    \end{eqnarray}
where $H$ denotes the Heaviside function and $$\begin{array}{cc}
    \delta_{h}(h\ul{m})  =  \left\{
    \begin{array}{ccc}
     \frac{1}{h^3} & \mbox{if} & h\ul{m}=  {\bf 0} \\
    0 & \mbox{if} & h\ul{m} \neq   {\bf 0}
    \end{array}
    \right., &     \delta_{\tau}(k\tau)  =  \left\{
    \begin{array}{ccc}
    \frac{1}{\tau} & \mbox{if} & k\tau= 0 \\
    0 & \mbox{if} & k\tau \neq   0
    \end{array}
    \right.,
    \end{array} $$ are the discrete analogues of the Dirac delta
    function in $\BR^3_h$ and $\BR_{\tau},$ respectively. Easy calculations show that, indeed,
    we have
\begin{equation}
{\small  (-i \partial_{\tau}-\Delta_h) e_{h,-i\tau} (h \ul m, k\tau)
= e_{h,-i\tau} (-i
\partial_{\tau}-\Delta_h) (h \ul m, k\tau)= \delta_\tau (k\tau) \delta_h (h \ul
m).}
\label{SchFundSolution}
    \end{equation}

By the factorization property (\ref{factH}), we have for the
discrete fundamental solution of the operator $D_{h,-i\tau}$ the
function
    \begin{eqnarray*}
    E_{h,-i\tau} & = & e_{h,-i\tau} D_{h,-i\tau}.
    \end{eqnarray*}

Moreover, straightforward calculations give the following matrix
representation for the discrete fundamental solution $E_{h,-i\tau}$
 \begin{gather*}
    E_{h,-i\tau}(h\ul{m},k\tau) =  \\
   {\small  \left[  \left(
    \begin{array}{cc}
      {\bf 0} & D^{-+}_h e_{h,-i\tau}  \\
      D^{+-}_h e_{h,-i\tau} & {\bf 0} \\
    \end{array} \right) +  \partial_{\tau}e_{h,-i\tau}  \left(
    \begin{array}{cc}
      \e_0 & {\bf 0} \\
       {\bf 0} &  \e_0  \\
    \end{array} \right) \gamma^+
    - i  e_{h,-i\tau} \left(
    \begin{array}{cc}
      \e_0 & {\bf 0} \\
       {\bf 0} &  \e_0  \\
    \end{array} \right) \gamma^- \right] }
    \end{gather*}

However, it remains to prove that the discrete fundamental solution
    $e_{h,-i\tau}$ is indeed an approximation of the fundamental
    solution (\ref{Form:2}). This will be done in the next section.


\section{Discrete operator calculus}

We define the discrete $l_p$-spaces,  $1 \leq p < \infty,$ in the
usual way
$$g \in l_p(\BR^3_h \times\BR^+_{\tau} )$$
iff $$||g||_{l_p(\BR^3_h \times \BR^+_{\tau})} = \left(
\sum_{(h\underline m, \tau k) \in \BR^3_h \times \BR^+_{\tau}} h^3
\tau |g(h\underline m, \tau k)|^p \right)^{\frac{1}{p}} < \infty .$$

Henceforward, no distinction will be made between the function $u :
\Omega \rightarrow \BC^4$ and its  restriction $u=u(h\ul{m},k\tau)$
to the lattice $\Omega_{h,\tau} = \Omega \cap (\BR^3_h \times
\BR^+_{\tau}),$ this distinction being clear from the context.

\subsection{Behavior of the discrete fundamental solution}

We now study the behavior of the discrete fundamental solution
(\ref{fundLapl}) when $h$ and $\tau$ tend to zero and we prove that
it converges in $l_1$-sense to the restriction to the grid of the
 fundamental solution (\ref{Form:2}).

\begin{theo} \label{Th:41}
Let $\frac{\tau}{h^2} < \frac{1}{6 \pi^2}.$ Then for any bounded
domain $G \subset \BR^3$ it holds
$$|| e_{h, -i\tau} - e_-  ||_{l_1(G_h \times [0, +\infty)_\tau)} \rightarrow
0$$ as $h, \tau \rightarrow 0.$
\end{theo}

The proof of this theorem is based on \cite{GH2}, Theorem 1, after
adaptation to space dimension $n=3$ and taking in account that our
solutions differ from the ones in the case of the heat operator by
the relations
$$e_-(\cdot, \cdot) = ie(\cdot, i \cdot)~~(continuous ~case)$$ and
$$e_{h,-i\tau}(\cdot, \cdot) = ie_{h,\tau}(\cdot, i \cdot), ~~(discrete
~case).$$ Moreover, due to the fact that the constructed discrete
fundamental solution $e_{h, -i\tau}$ has a conical support domain we
obtain the mesh-size condition  $\frac{\tau}{h^2} < \frac{1}{6
\pi^2}.$

We remark that Theorem \ref{Th:41} implies the
$l_1^{loc}$-convergence of (\ref{fundLapl}) to (\ref{Form:2}). Also,
as an immediate consequence we have
\begin{cor}
Under the conditions of Theorem \ref{Th:41} it holds
$$|| E_{h, -i\tau} - E_-  ||_{l_1(G_h \times [0, +\infty)_\tau)} \rightarrow
0$$ for any bounded discrete domain $G_h \subset \BR^3,$ as $h, \tau
\rightarrow 0.$
\end{cor}

While we can prove the convergence of the discrete solution $E_{h,
-i\tau}$ to $E_-,$ the proofs do not yield the order of convergence
due to the nature of the continuous fundamental solution of the
Schr\"odinger equation. This will be the subject of future work.

Hence, we can establish the discrete analogues of the Teodorescu
operator.

\begin{theo} For all $u \in l_p(\Omega_{h, \tau}),$ $1<p<+\infty,$ such that $u : \Omega_{h, \tau} \rightarrow \BC^4$ we
have the discrete Teodorescu operator $T_{h,-i\tau}$ satisfying to
\begin{equation} D_{h,-i\tau} T_{h,-i\tau}u (h\ul{m},k\tau) =
u (h\ul{m},k\tau),  \label{inversodireitaD-+}\end{equation}
 where
 \begin{equation}
    T_{h,-i\tau}u  (h\ul{m},k\tau) = \sum_{( h \ul n, s \tau)
    \in ~ \Omega_{h,\tau}} h^{3} \tau
    E_{h,-i\tau}(h\ul{m}-h \ul n, k\tau - s\tau) u( h \ul n ,s \tau), \label{Def:T-+}
\end{equation}
for all $(h\ul{m},k\tau)\in \Omega_{h,\tau} $.
\end{theo}

\begin{proof}
We have for $T_{h,-i\tau}$ that
   \begin{gather*}   D_{h,-i\tau} T_{h,-i\tau}u (h\ul{m},k\tau) = \sum_{( h \ul n, s \tau)
    \in ~ \Omega_{h,\tau}} h^{3} \tau
    [ D_{h,-i\tau} E_{h,-i\tau}](h\ul{m}-h \ul n, k\tau - s\tau) u( h \ul n ,s
    \tau).
        \end{gather*}
Since $E_{h,-i\tau} = e_{h,-i\tau}D_{h,-i\tau}$ and $e_{h,-i\tau}$
is a scalar solution, we have
       \begin{gather*}
   D_{h,-i\tau} T_{h,-i\tau}u (h\ul{m},k\tau)   \\
    = \sum_{( h \ul n, s \tau)
    \in ~ \Omega_{h,\tau}} h^{3} \tau
     [ e_{h,-i\tau}(D_{h,-i\tau} )^2 (h\ul{m}-h \ul n, k\tau - s\tau) ] u( h \ul n ,s
    \tau) \\
       = \sum_{( h \ul n, s \tau)
    \in ~ \Omega_{h,\tau}} h^{3} \tau ~[     \delta_h(h\ul{m} - h \ul n)  \delta_{\tau}(k \tau - s \tau)u( h \ul n ,s
    \tau)] \\
    = u( h \ul{m},k\tau).
    \end{gather*}
\end{proof}

Now we are able to present the following norm estimate.
\begin{theo}\label{Th:42}
For all $u \in l_p(\Omega_{h,\tau}),$ $1<p<+\infty,$ such that $u :
\Omega_{h, \tau} \rightarrow \BC^4$ there exists a positive constant
$C>0$ such that
    \begin{eqnarray*}
    ||T_{h,-i\tau}u  ||_{l_p(\Omega_{h,\tau})} & \leq &  C
    ||u||_{l_p(\Omega_{h,\tau})}.
    \end{eqnarray*}
Moreover, $T_{h,-i\tau}$ is a continuous operator.
\end{theo}

\begin{proof}
Initially we have
    \begin{gather*}
    ||T_{h,-i\tau}u  ||_{l_p(\Omega_{h,\tau})} = \\ \\
    = \left( \sum_{(h\ul n, s \tau) \in \Omega_{h,\tau}} \tau
    h^3 \left| E_{h,-i\tau}(h\ul{m}-h \ul n, k\tau - s\tau) u( h \ul n ,s \tau)
    \right|^p
    \right)^{\frac{1}{p}} \\ \\
    \leq
    \left( \sum_{(h\ul n, s \tau) \in \Omega_{h,\tau}} \tau
     h^3 \left|E_{h,-i\tau}(h\ul{m}-h \ul n, k\tau  - s\tau)
    \right|^p     \left| u( h \ul n ,s \tau) \right|^p
    \right)^{\frac{1}{p}}.
    \end{gather*}
Let us take $C(\ul{m}, k) = \max_{(h\ul{n}, s \tau) \in \Omega_{h,
\tau}}|E_{h,-i\tau}(h\ul{m} - h\ul{n}, k\tau  - s \tau)|.$ Then
there exists $C = \max C(\ul{m}, k) >0,$ this maximum being taken
over all $(\ul{m}, k)$ such that $(h\ul{m}, k\tau) \in \Omega_{h,
\tau},$ and the result holds.
\end{proof}

As we have done for the analytic case we can establish a
decomposition of the $l_p$-space.

\begin{theo}\label{dd}
For the space $l_p(\Omega_{h,\tau}),$ $1<p<\infty,$ the following
direct decomposition
    \begin{eqnarray*}
   l_p(\Omega_{h,\tau}) & = & \ker D_{h,-i\tau}(\mbox{\rm int}
    \Omega_{h,\tau}) \oplus D_{h,-i\tau}({\stackrel{\circ}{{w}_{p}^1}} (\Omega_{h,\tau}))
    \end{eqnarray*}
is valid, with correspondent discrete projection operators
    \begin{eqnarray*}
   P_{h,\tau} & : & l_p(\Omega_{h,\tau}) \mapsto \ker D_{h,-i\tau}
    (\mbox{\rm int} \Omega_{h,\tau}), \\
    Q_{h,\tau} & : & l_p(\Omega_{h,\tau}) \mapsto D_{h,-i\tau}
    ({\stackrel{\circ}{{w}_{p}^1}}(\Omega_{h,\tau})),
    \end{eqnarray*} where
    ${\stackrel{\circ}{{w}_{p}^1}}(\Omega_{h,\tau})$ denotes the
    discrete counterpart of the Sobolev space ${\stackrel{\circ}{{W}_{p}^1}}(\Omega).$
\end{theo}

\subsection{Convergence of the discrete operators}

We say that $u \in C^{1,\alpha}(\Omega)$ if its first derivatives
are $\alpha$-H\"older continuous.

\begin{theo}
Let $u \in C^{1,\alpha}(\Omega).$ Then it holds $T_{h,-i\tau} u
\rightarrow T u$ as $h, \tau $ tend to zero.
\end{theo}

\begin{proof} In order to prove the above result we introduce the
regularized Teodorescu operator (see \cite{Tao})
$$T^\varepsilon u (x, t) = \int_\Omega E_-^\varepsilon (x-z,t-r) u(z,r)dz dr,$$
where
$$E_-^\varepsilon (x,t) = e^{-\epsilon \frac{|x|^2}{4t}}E_-(x,t)$$
stands for a regularization of the fundamental (continuous) solution
$E_-$ and, therefore, it converges in the sense of tempered
distributions to $E_-$ as $\varepsilon \rightarrow 0.$ In a similar
way, we construct the regularized discrete operator
$T_{h,-i\tau}^{\varepsilon}$ in terms of the discrete analogue of the regularized
fundamental solution $$E^{\varepsilon}_{h,-i\tau} = e^{ -\epsilon
\frac{|h \underline m|^2}{4k \tau} }  E_{h,-i\tau}.$$

By definition, we have
\begin{gather}
 |T_{h,-i\tau}^{\varepsilon} u (h \underline m, k \tau) - T^\varepsilon u (h \underline m, k
 \tau)|  \nonumber \\
 \leq     \left| \sum_{(h \ul n,s \tau) \in \Omega_{h, \tau}} E^{\varepsilon}_{h,-i\tau}(h \ul m-h \ul n, k \tau-s \tau) u(h \ul n, s \tau) h^3 \tau
  \right. \nonumber \\
  -  \left.     \int_{\Omega} E_-^\varepsilon(h \underline m-z,k \tau-r) u(z,r)
  dz    dr     \right|.  \label{gather1}
     \end{gather}

Due to the singularity of the continuous fundamental solution $E_-^\varepsilon$,
we will split the continuous domain $\Omega$ into parallelepiped
$W(h \ul n, s \tau)$ centered at the points $(h \ul n, s \tau )$ of
the lattice $\Omega_{h, \tau}$ with side-lengths $h$ and $\tau$,
respectively. Furthermore, let $p,q \in \BN$ be such that
$\frac{1}{p}+\frac{1}{q}=1.$ We have then
\begin{gather}
(\ref{gather1}) \leq  \nonumber \\
\sum_{(h \ul n,s \tau) \in \Omega_{h, \tau}}
 \left| [ E^{\varepsilon}_{h,-i\tau}(h \ul m-h \ul n, k \tau-s \tau) - E_-^\varepsilon(h \ul m-h \ul n, k \tau-s \tau)] u(h \ul n, s \tau)
h^3 \tau   \right| \nonumber \\
+ \left|  \sum_{(h \ul n,s \tau) \in \Omega_{h, \tau}}   \left[
E_-^\varepsilon(h \ul m-h \ul n, k \tau-s \tau) u(h \ul n, s \tau)
h^3 \tau
\right. \right. \nonumber \\
- \left.   \int_{W(h \ul n, s \tau)} E_-^\varepsilon(h \underline
m-z,k \tau-r) u(z,r) dz dr    ~]  \right|. \label{FirstEstimate}
     \end{gather} We use H\"older's inequality on the first
     term and by a convenient adding up we get
\begin{gather*}
  (\ref{FirstEstimate})  \leq  ||E^{\varepsilon}_{h,-i\tau}  -  E_-^\varepsilon ||_{l_p(\Omega_{h,\tau}) } ||u||_{l_q(\Omega_{h,\tau}) } \\
       + \underbrace{ \sum_{(h \ul n,s \tau) \in \Omega_{h, \tau}, ~z \in W(h \ul n,s \tau)}
 \left[ \left| ~ E_-^\varepsilon(h \ul m-h \ul n, k \tau-s \tau) [ u(h \ul n, s \tau) -
 u(z, r) ] h^3 \tau \right| \right. }_{(I_1)}  \nonumber \\
+   \underbrace{ \int_{W(h \ul n, s \tau)} \left|
\left[E_-^\varepsilon(h \ul m-h \ul n, k \tau-s \tau) -
E_-^\varepsilon(h \underline m-z,k \tau-r) \right] u(z,r) \right| dz
dr }_{(I_2(h \ul n, s \tau))}  ~] .
    \end{gather*}

For the term $(I_1)$ we obtain
\begin{gather*}
  (I_1)  \leq \sum_{(h \ul n,s \tau) \in \Omega_{h, \tau}, ~z \in W(h \ul n,s
  \tau)} | E_-^\varepsilon(h \ul m-h \ul n, k \tau-s \tau)| \\
\times  \int_{W(h \ul n,s \tau)} |u(h \ul n , s \tau) - u(z,r)| dz dr \\
 \leq \sum_{(h \ul n,s \tau) \in \Omega_{h, \tau}, ~z \in W(h \ul n,s
  \tau)} | E_-^\varepsilon(h \ul m-h \ul n, k \tau-s \tau)| \\
  \times  C \int_{W(h \ul n,s \tau)} |(h \ul n - z,s \tau-r)|^\alpha dz dr,
    \end{gather*}
which goes to zero as $h, \tau \rightarrow 0.$

Finally the term $(I_2(h \ul n, s \tau))$ can be estimate using its
Taylor series expansion and H\"older's inequality
    \begin{gather*}
  (I_2(h \ul n, s \tau))  \leq  \int_{W(h \ul n, s \tau)} \left|
\left[ E_-^\varepsilon(h \ul m-h \ul n, k \tau-s \tau) -
E_-^\varepsilon(h \underline m-z,k \tau-r) \right] u(z,r) \right| dz
dr \\
\leq   \int_{W(h \ul n, s \tau)} \left| \nabla E_-^\varepsilon(h \ul
m- z, k \tau- r) \cdot (h \underline n- z,s  \tau- r )
\right| |u(z,r)| dz dr \\
\leq || \nabla E_-^\varepsilon(h \ul m- \cdot, k \tau- \cdot) \cdot
(h \underline n-\cdot,s  \tau-\cdot )||_{L_q(W(h \ul n, s \tau))}
||u||_{L_p(W(h \ul n, s \tau))},
    \end{gather*} and  again we have that
    $\sum_{(h \ul n,s \tau) \in \Omega_{h, \tau}, ~z \in W(h \ul n,s
  \tau)} (I_2(h \ul n, s \tau)) $ goes to zero as $h, \tau \rightarrow 0.$

Hence, by $\varepsilon \rightarrow 0$ we obtain convergence of the
discrete Teodorescu operator $T_{h, -i\tau}$ to the continuous one.
\end{proof}

Moreover, we notice that we have convergence in $l_p,1<p<\infty,$ of
the regularized discrete Teodorescu operator $T_{h,
-i\tau}^\epsilon$ to the regularized continuous operator
$T^\epsilon.$


We now prove the convergence of the discrete Cauchy-Bitsadze
operator $F_{h,-i\tau} = I - T_{h,-i\tau} D_{h,-i\tau}.$
 Moreover, in what follows we will consider the sub-domains
$    \Omega^t = \left\{ x \in \BR^3: (x,t) \in \Omega \right\}$ and
$\Omega^x = \left\{ t \in \BR^+: (x,t) \in \Omega
    \right\}.$

\begin{theo}\label{fh}
If $u\in \ker D_{x,-it}$ is such that $u  \in C^{1,\alpha}(\Omega)$
for some $0<\alpha< 1$  then  we have
    \begin{eqnarray*}
    ||u- F_{h,-i\tau}u||_{l_p(\Omega_{h,\tau})} & \leq & C ||u||_{C^{1,\alpha}(\Omega)}
    (h^\alpha +\tau^\alpha),
    \end{eqnarray*} for a positive constant $C>0.$
\end{theo}

\begin{proof}
We use the definition of $F_{h,-i\tau},$ Theorem \ref{Th:42} and the
fact that $u \in \ker D_{x,-it}.$ We get then
    \begin{eqnarray}
    ||u-F_{h,-i\tau}u||_{l_p(\Omega_{h,\tau})} & = & ||T_{h,-i\tau}D_{h,-i\tau}u ||_{l_p(\Omega_{h,\tau})} \nonumber \\
    & = & ||T_{h,-i\tau}(D_{h,-i\tau} u - D_{x,-it}u)||_{l_p(\Omega_{h,\tau})} \nonumber \\
    & \leq & C_1 ||D_{h,-i\tau} u - D_{x,-it} u
    ||_{l_p(\Omega_{h,\tau})} \nonumber \\
& \leq & C_1  \left( ||D_{h} u  - D_{x } u ||_{l_p(\Omega_{h,
\tau})} + ||\partial_\tau u  - \partial_t u ||_{l_p(\Omega_{h,
\tau})} \right)
\nonumber \\
& \leq & C_1 \left[ \left( \sum_{(h \ul m, k \tau) \in \Omega_{h,
\tau} } |D_h u (h \ul m, k \tau) - D_x u(h \ul m, k \tau) |^p
h^3 \tau \right)^{\frac{1}{p}}  \right. \nonumber \\
& & \left.    +  \left( \sum_{(h \ul m, k \tau) \in \Omega_{h,
\tau}} |\partial_\tau u(h \ul m, k \tau) -
\partial_t u (h \ul m, k \tau) |^p h^3
\tau \right)^{\frac{1}{p}}  \right].  \label{aux2}
    \end{eqnarray}

Additionally, we remark that $ u  \in C^{1,\alpha}(\Omega)$ implies
both
$$u(\cdot, t)  \in C^{1,\alpha}(\Omega^t),~~~u(x,\cdot) \in C^{1,\alpha}(\Omega^x).$$
Moreover, we have (c.f. \cite{GS}, p.268) that
\begin{equation}
|D_{h}^{+-} u (h \ul m, k\tau)- D_{x } u (h \ul m, k\tau) | \leq K(k
\tau) ||u(\cdot, k \tau) ||_{C^{1,\alpha}(\Omega^{k \tau})}
h^\alpha, \label{rel1}
\end{equation} a similar result holding for $D_{h}^{-+},$ and
\begin{equation}
|\partial_\tau  (h \ul m, k\tau)- \partial_t  u (h \ul m, k\tau) |
\leq K(h \ul m) ||u(h \ul m, \cdot) ||_{C^{1,\alpha}(\Omega^{h \ul
m})} \tau^\alpha, \label{rel2}
\end{equation} for some positive constants $K(k \tau), K(h \ul m).$
 Using these two inequalities we have
    \begin{gather*}
(\ref{aux2}) \leq C_1 \left[ \left( \sum_{(h \ul m, k \tau) \in
\Omega_{h, \tau} } K^p(k \tau)  ||u(\cdot, k \tau) ||_{C^{1,\alpha}(\Omega^{k \tau})}^p h^{p \alpha} h^3 \tau \right)^{1/p}  \right.   \\
\left.    +  \left( \sum_{(h \ul m, k \tau) \in \Omega_{h, \tau} }
 K^p(h \ul m)  ||u(h \ul m, \cdot) ||_{C^{1,\alpha}(\Omega^{h \ul
 m})}^p \tau^{p \alpha}  h^3 \tau \right)^{1/p}  \right].
    \end{gather*}

We now take $K = \max_{\Omega_{h,\tau}} \{ K(k \tau), K(h \ul m) \}
>0$ and we recall that $$ ||u(h \ul m, \cdot)
||_{C^{1,\alpha}(\Omega^{h \ul m})}    \leq
||u||_{C^{1,\alpha}(\Omega)} ,~~||u(\cdot, k \tau)
||_{C^{1,\alpha}(\Omega^{k \tau})} \leq
||u||_{C^{1,\alpha}(\Omega)}. $$

Hence
    \begin{gather*}
(\ref{aux2}) \leq C_1 K Vol(\Omega_{h,\tau})
 ||u||_{C^{1,\alpha}(\Omega)} (h^\alpha + \tau^\alpha) .
    \end{gather*}

\end{proof}

We are now in conditions to prove the convergence of the discrete
projection operator $Q_{h,\tau}$ to its continuous counterpart
(\ref{cont_Q}).

\begin{theo}
Let $u  \in L_p(\Omega)$ for some $ 1< p < \infty.$  Then it holds
for the projector $Q_{h,\tau}$
    \begin{eqnarray*}
    ||Q_{h,\tau} u - Q u||_{l_p(\Omega_{h,\tau})}  \rightarrow 0 & as &
    h, \tau ~\rightarrow 0
    \end{eqnarray*} for a positive constant $C.$
\end{theo}

\begin{proof}
We start from the equality
    \begin{eqnarray*}
    Q_{h,\tau} u - Q u & = & Q_{h,\tau} ( P u + Q u) - Q ( P u
    +
    Q u) \\
& = & Q_{h,\tau} P u + Q_{h,\tau}  Q u  - Q u
    \end{eqnarray*}
and we wish to obtain estimates for the terms $Q_{h,\tau} P u$ and
$(Q_{h,\tau} - I )Q u $ (we recall that, being projection operators,
$Q(Pu) =0$ and $Q^2 = Q$).

Since $P u= F P u$ and $Q_{h,\tau} F_{h,-i\tau} u = 0$, for the
first term we obtain
    \begin{eqnarray*}
    Q_{h,\tau} P u & = & Q_{h,\tau} F P u - Q_{h,\tau} F_{h,-i\tau} P u \\
    & = & Q_{h,\tau}(F - F_{h,-i\tau}) P u \\
    & = & Q_{h,\tau}(I - F_{h,-i\tau} - T D_{x,-it}) P u \\
    & = & Q_{h,\tau}(I - F_{h,-i\tau}) P u
    \end{eqnarray*}
and, therefore, by Theorem \ref{fh} we get the following estimate
    \begin{eqnarray*}
    ||Q_{h,\tau} P u||_{l_p(\Omega_{h,\tau})} & \leq &
    ||Q_{h,\tau}|| ~ ||P u - F_{h,-i\tau} P u ||_{l_p(\Omega_{h,\tau})} \\
    & \leq & C ||Q_{h,\tau}|| ~ ||P u||_{C^{1, \alpha}(\Omega)} (h^\alpha +
    \tau^\alpha),
    \end{eqnarray*} taking in account that $Q_ {h,\tau}$ has bounded norm. Moreover, due to
    the fact that $P$ is the projection into the kernel of $D_{h,
    -i\tau},$ it holds $||P u||_{C^{1, \alpha}(\Omega)}<\infty.$

For the second term we remember that $Q u$ can be written as $Q u =
D_{x,-it} g$ where $g \in \stackrel{\circ}{{\cW}^{1}_2} (\Omega).$
This leads to
    \begin{eqnarray*}
    (Q_{h,\tau} - I )Q u  & = & (Q _{h,\tau} - I )D_{x,-it} g \\
    & = & Q_{h,\tau} (D_{x,-it} g - D_{h,-i\tau} g ) + Q_{h,\tau} D_{h,-i\tau} g - D_{x,-it} g \\
    & =& Q_{h,\tau} (D_{x,-it} g - D_{h,-i\tau} g ) + (D_{h,-i\tau} g - D_{x,-it} g),
    \end{eqnarray*}
since  $Q_{h,\tau} D_{h,-i\tau} g = D_{h,-i\tau}g$. Hence, taking
into account the previous calculations, Theorem \ref{fh} and
relations (\ref{rel1}) and (\ref{rel2}) we finally obtain
    \begin{eqnarray*}
    ||(Q_{h,\tau} - I )Q u||_{l_p(\Omega_{h,\tau})} & \leq &
( ||Q_{h,\tau}|| +1) ||D_{h,-i\tau} g - D_{x,-it} g
||_{l_p(\Omega_{h,\tau})} \rightarrow 0
    \end{eqnarray*} as $h, \tau$ goes to zero.
\end{proof}

The above discrete operators allow us to establish a discrete
equivalent of Theorem \ref{Th:26}.

\begin{theo} \label{Th:28}Let $f \in l_2(\Omega_{h,\tau}).$ The solution of the discrete Schr\"odinger problem
    $$\left \{
    \begin{array}{rcl}
 (i  \partial_\tau  - \Delta_h) u & = & f \mbox{ in } \Omega_{h,\tau}  \\
    u & = & 0 \mbox{ on } \partial \Omega_{h,\tau}
    \end{array}
    \right. $$
is given by $u=T_{h,-i\tau}Q_{h,\tau}T_{h,-i\tau} f.$
\end{theo}


\section{The non-linear Schr\"odinger problem}

Let us now consider the non-linear Schr\"odinger problem
\begin{equation}\left \{
    \begin{array}{cccc}
    i\partial_t u-\Delta u  & = & M(u) & \mbox{ in }
    \Omega   \nonumber \\
    u & = & 0 & \mbox{ on }\partial \Omega
    \end{array}
    \right. \label{nonlinearSchrodinger} \end{equation}
where $M(u)=|u|^2u+f,$ with $f \in L_2(\Omega),$ and $
|u|^2=\sum^3_{j=0}(u^j)^2$. This problem can be reduced to
\begin{equation}   u  = T Q T  M(u )  \mbox{ in }
    \Omega,\label{metcont}
    \end{equation} a problem for which the next theorem proves existence and uniqueness of solution (see \cite{B}, \cite{CV} for details).

\begin{theo}\label{Th:5.2}
The problem (\ref{nonlinearSchrodinger}) has an unique solution
given in terms of the iterative method $$u_{n+1} = TQT M(u_n)$$ if
$f \in L_2(\Omega)$ satisfies the condition
    \begin{eqnarray*}
    ||f||_{L_2} & \leq & \frac{1}{36 \cdot 2^{m+1}}.
    \end{eqnarray*}
Moreover, the iteration method converges for each starting point
$u_0 \in \stackrel{\circ}{W}_2^{1}(\Omega)$ such that
    \begin{eqnarray*}
    ||u_0||_{L_2} & \leq & \frac{1}{6 \cdot 2^{m+1}} + W,
    \end{eqnarray*}
with $W = \sqrt{\frac{1}{36 \cdot
2^{2(m+1)}}-\frac{||f||_{L_2}}{2^{m+1}}}$.
\end{theo}

Based on the discrete operators previously introduced we construct
the discrete version of problem (\ref{metcont}) for our bounded
domain
\begin{equation}   u  = T_{h,-i\tau} Q_{h,\tau} T_{h,-i\tau}  M(u )  \mbox{ in }
    \Omega_{h,\tau}. \label{met}
    \end{equation}

Indeed, let $v$ be a  solution of (\ref{met}). Then
\begin{eqnarray*}
(i \partial_\tau - \Delta_h) v &=& D_{h,-i\tau}D_{h,-i\tau}
[T_{h,-i\tau} Q_{h,\tau} T_{h,-i\tau}  M(v ) ]\\
&=& D_{h,-i\tau} [ Q_{h,\tau} T_{h,-i\tau}  M(v ) ]\\
& = & M(v),
\end{eqnarray*}
and due to the properties of the projector $Q_{h,\tau}$ we have
$v=0$ on $\partial\Omega_{h,\tau}$.

Using the same ideas as in the continuous case (see \cite{CV}) we
get results regarding the convergence and uniqueness of the discrete
iterative method $u_{n+1}= T_{h,-i\tau} Q_{h,\tau} T_{h,-i\tau}
M(u_n).$

\begin{theo} If $f \in l_{2}(\Omega_{h,\tau})$ then the discrete problem (\ref{met}) has a
unique solution $u \in \stackrel{\circ}{w_{2}^1}(\Omega_{h,\tau})$
whenever
    \begin{eqnarray*}
    ||f||_{l_2(\Omega_{h,\tau})} & \leq & \frac{1}{36C_{h,\tau}}
    \end{eqnarray*}
and the initial term $u_0 \in \stackrel{\circ}{w_{2}^1}
(\Omega_{h,\tau})$ satisfies
    \begin{eqnarray*}
    ||u_0||_{l_2(\Omega_{h,\tau})} &
    \leq & \frac{1}{6C_{h,\tau }} + W_{h,\tau},
    \end{eqnarray*}
with $W_{h,\tau} = \sqrt{ \frac{1}{36C_{h,\tau}} -
\frac{||f||_{l_2(\Omega_{h,\tau})}}{C_{h,\tau}}}$.
\end{theo}

The proof of this theorem, being similar to the one in the
continuous case, will be omitted.


The following result shows that the solution obtained for the
discrete problem, which we will denote by $u_\ast$, converges to the
solution obtained for the continuous, which we will denote by $u$.
In the proof of the following theorem the restriction of $M(u)$ to
the space-time grid will be denote by $M_{h,\tau}(u)$.

\begin{theo}
Let $f \in L_2(\Omega)$. Then $u_\ast$ converges to $u$ in
$\Omega_{h,\tau}$ whenever $h,\tau \rightarrow 0$.
\end{theo}
\begin{proof}
Again, we need to use the regularized Teodorescu operator. We shall
denote $u_\ast^\epsilon =  T_{h,-i\tau}^\epsilon
Q_{h,\tau}T_{h,-i\tau}^\epsilon M_{h,\tau}(u_\ast^\epsilon) $ and
$u^\epsilon = T^\epsilon QT^\epsilon M(u^\epsilon).$ We have
    \begin{eqnarray*}
    || u_\ast^\epsilon-u^\epsilon ||_{l_2(\Omega_{h,\tau})}
    & \leq & \underbrace{|| T_{h,-i\tau}^\epsilon Q_{h,\tau}T_{h,-i\tau}^\epsilon M_{h,\tau}(u^\epsilon)
    - T^\epsilon QT^\epsilon M(u^\epsilon) ||_{l_2(\Omega_{h,\tau})}}_{(\mathbf{I})} \\
    & & ~~ + || T_{h,-i\tau}^\epsilon Q_{h,\tau}T_{h,-i\tau}^\epsilon\left(M_{h,\tau}(u^\epsilon_\ast)-
    M_{h,\tau}(u^\epsilon)\right) ||_{l_2(\Omega_{h,\tau})} \\
        & \leq &  (\mathbf{I}) + C_{h,\tau} || u^\epsilon_\ast -
    u^\epsilon ||_{l_2(\Omega_{h,\tau})} \left( \left|\left|u^\epsilon_\ast\right|\right|_{l_2(\Omega_{h,\tau})}
    +\left|\left|u^\epsilon \right|\right|_{l_2(\Omega_{h,\tau})} \right)
    \end{eqnarray*}
which implies that
    \begin{eqnarray*}
    || u^\epsilon_\ast-u^\epsilon ||_{l_2(\Omega_{h,\tau})}
    & \leq &  (\mathbf{I}) \left[ 1- C_{h,\tau} \left( \left|\left|u^\epsilon_\ast \right|\right|_{l_2(\Omega_{h,\tau})}
    +\left|\left|u^\epsilon \right|\right|_{l_2(\Omega_{h,\tau})} \right)
    \right]^{-1},
    \end{eqnarray*}
where $C_{h,\tau}$ is a positive constant which depends from $h$ and
$\tau$. By Theorem 5.2 we can guarantee that
    \begin{eqnarray*}
    \left|\left|u^\epsilon_\ast\right|\right|_{l_2(\Omega_{h,\tau})} &
    \leq & \frac{1}{6C_{h,\tau}} + W_{h,\tau},
    \end{eqnarray*}
with
$W_{h,\tau}=\sqrt{\frac{1}{36C_{h,\tau}}-\frac{||f||_{l_2(\Omega_{h,\tau})}}{C_{h,\tau}}}$.

This inequality, together with Theorem \ref{Th:5.2}, ensures that
for sufficiently small $h$ and $\tau$, the following relation
    \begin{eqnarray*}
    1- C_{h,\tau}
    \left( \left|\left|u^\epsilon_\ast\right|\right|_{l_2(\Omega_{h,\tau})}
    +\left|\left|u^\epsilon \right|\right|_{l_2(\Omega_{h,\tau})} \right)& > &
    0
    \end{eqnarray*}
holds. Therefore, the convergence of $u_\ast$ to $u$ depends only on
the term ($\mathbf{I}$). Hereby, we have
\begin{gather*}
    (\mathbf{I})= || T_{h,-i\tau}^\epsilon Q_{h,\tau}T_{h,-i\tau}^\epsilon
    M_{h,\tau}(u^\epsilon) - QT^\epsilon M(u^\epsilon)||_{l_2(\Omega_{h, \tau})} \\
     \leq  \underbrace{|| T_{h,-i\tau}^\epsilon  Q_{h,\tau} T_{h,-i\tau}^\epsilon
    \left( M_{h,\tau}^\ast(u^\epsilon) - M^\ast(u^\epsilon) \right)||_{l_2(\Omega_{h, \tau})}}_{(\mathbf{A})} \\
    +     \underbrace{|| T_{h,-i\tau}^\epsilon  Q_{h,\tau} \left(T_{h,-i\tau}^\epsilon - T^\epsilon \right)
    M^\ast(u^\epsilon)||_{l_2(\Omega_{h, \tau})}}_{(\mathbf{B})} + \underbrace{|| T_{h,-i\tau}^\epsilon  \left(Q_{h,\tau}-Q\right)
    T^\epsilon M^\ast(u^\epsilon)||_{l_2(\Omega_{h, \tau})}}_{(\mathbf{C})} \\
    +\underbrace{|| T_{h,-i\tau}^\epsilon Q_{h,\tau} \left(T_{h,-i\tau}^\epsilon - T^\epsilon\right)f
    ||_{l_2(\Omega_{h, \tau})}}_{(\mathbf{D})}+ \underbrace{|| T_{h,-i\tau}^\epsilon \left(Q_{h,\tau}
    -Q\right)T^\epsilon f||_{l_2(\Omega_{h, \tau})}}_{(\mathbf{E})},
    \end{gather*}
where $M^\ast(u^\epsilon)=|u^\epsilon|^2u^\epsilon$ and
$M_{h,\tau}^\ast(u^\epsilon)$ denotes its restriction to the
space-time grid. By Theorem 4.6 we can say that ($\mathbf{B}$) and
($\mathbf{D}$) tend to zero as $h, \tau \rightarrow 0 .$ Also,
Theorem 4.8 implies the same result for both ($\mathbf{C}$) and
($\mathbf{E}$). Finally, for ($\mathbf{A}$) we have, from the
boundedness of the discrete operators, the following relation
    \begin{gather*}
    || T_{h,-i\tau}^\epsilon Q_{h,\tau} T_{h,-i\tau}^\epsilon \left( M_{h,\tau}^\ast(u^\epsilon) - M^\ast(u^\epsilon)
    \right)||_{l_2(\Omega_{h, \tau})} \\
     \leq  || T_{h,-i\tau}^\epsilon  Q_{h,\tau}     T_{h,-i\tau}^\epsilon||_{l_2(\Omega_{h, \tau})} ||M_{h,\tau}^\ast(u^\epsilon) - M^\ast(u^\epsilon)||_{l_2(\Omega_{h, \tau})} \\
     \leq  C_1 C_{h,\tau},
    \end{gather*}
where $C_1$ is a finite constant and $C_{h,\tau}$ is a constant
which depends on $h$ and $\tau$ and goes to zero with $h$ and
$\tau.$ Therefore, ($\mathbf{I}$) tends to zero when $h,\tau
\rightarrow 0$, thus, proving our result as $\epsilon\rightarrow 0$.
\end{proof}

\section{Numerical Examples}

In order to study the rate of convergence of our method for
different mesh sizes, we shall present some numerical examples. For
simplicity sake, we shall use a cubic space domain $[-a,a]^3$ with
an equidistant discretization grid of $(N+1)^3$ points. Also, for
the discretization of the time domain we shall consider an
equidistant grid with M+1 mesh-points. At this point, we emphasize
that the choice of $M$ and $N$ takes into account the restriction
$\frac{\tau}{h^2} < \frac{1}{6\pi^2}$ imposed by Theorem 4.1.

For all the examples below we will be presenting a table with the
$l^1-$error between the approximated solution and the exact solution
at given instants of time.\\

\textbf{Example 1:} As a first example, we consider an exact
real-valued $C^{\infty}$ solution $ u = (0,u_1,u_2,u_3)$ for the
problem (\ref{nonlinearSchrodinger}), where
    \begin{gather*}
    u_1(x,t) = e^{-x_1}~\cos\left( \pi t + \frac{\pi}{2} \right)~\sin(\pi x_1x_2x_3) \\
    u_2(x,t) =  u_3(x,t) =0,
    \end{gather*}
and the corresponding right hand side $f = i \partial_t u -\Delta u
- |u^2|u.$

In the following table we show the approximation error between the
exact solution $u$ and its discrete approximation $u_{h,\tau}$ on
the domain $\Omega = [-5,5]^3\times[0,2]$ for different mesh sizes.

\begin{center}\begin{small}
\begin{tabular}{|c|c|c|c|c|}
  \multicolumn{5}{c}{\textbf{Table 1}} \\
  \hline \hline
  N & M & t=0 & t=0.4 & t=0.8 \\
  \hline
  20 & 450   & 2.3313$\times 10^{-3}$ & 1.2799$\times 10^{-3}$ & 5.7386$\times 10^{-4}$  \\
  25 & 703   & 1.5265$\times 10^{-3}$ & 8.3774$\times 10^{-4}$ & 3.7642$\times 10^{-4}$  \\
  30 & 1013  & 1.0765$\times 10^{-3}$ & 5.9073$\times 10^{-4}$ & 2.6569$\times 10^{-4}$  \\
  35 & 1378  & 7.9982$\times 10^{-4}$ & 4.3844$\times 10^{-4}$ & 1.9706$\times 10^{-4}$  \\
  40 & 1800  & 6.1732$\times 10^{-4}$ & 3.3895$\times 10^{-4}$ & 1.5228$\times 10^{-4}$  \\
  45 & 2278  & 4.9075$\times 10^{-4}$ & 2.6919$\times 10^{-4}$ & 1.2107$\times 10^{-4}$  \\
  50 & 2813  & 3.9937$\times 10^{-4}$ & 2.1923$\times 10^{-4}$ & 9.8534$\times 10^{-5}$  \\
  55 & 3404  & 3.3132$\times 10^{-4}$ & 1.8193$\times 10^{-4}$ & 8.1714$\times 10^{-5}$  \\
  \hline \hline
  N & M & t=1.2 & t=1.6 & t=2 \\
  \hline
  20 & 450   & 2.5728$\times 10^{-4}$ & 1.1633$\times 10^{-4}$ & 5.3040$\times 10^{-5}$ \\
  25 & 703   & 1.6914$\times 10^{-4}$ & 7.5998$\times 10^{-5}$ & 3.4520$\times 10^{-5}$ \\
  30 & 1013  & 1.1950$\times 10^{-4}$ & 5.3548$\times 10^{-5}$ & 2.4266$\times 10^{-5}$ \\
  35 & 1378  & 8.8572$\times 10^{-5}$ & 3.9810$\times 10^{-5}$ & 1.7992$\times 10^{-5}$ \\
  40 & 1800  & 6.8416$\times 10^{-5}$ & 3.0738$\times 10^{-5}$ & 1.3868$\times 10^{-5}$ \\
  45 & 2278  & 5.4362$\times 10^{-5}$ & 2.4450$\times 10^{-5}$ & 1.1014$\times 10^{-5}$ \\
  50 & 2813  & 4.4226$\times 10^{-5}$ & 1.9878$\times 10^{-5}$ & 8.9580$\times 10^{-6}$ \\
  55 & 3404  & 3.6700$\times 10^{-5}$ & 1.6502$\times 10^{-5}$ & 7.4280$\times 10^{-6}$ \\
  \hline \hline
  \multicolumn{5}{c}{$l_1-$error between the approximated solution and the exact solution}\\
  \multicolumn{5}{c}{at different instants}
\end{tabular}
\end{small}\end{center}

The following graphics (Figures 1. and 2.) show the evolution of the
$l_1-$norm for the approximation error, with respect to the
space-mesh and to the time-mesh, respectively.
\begin{figure}[htb]
    \begin{minipage}[h]{0.5\linewidth}
    \includegraphics[width=\linewidth]{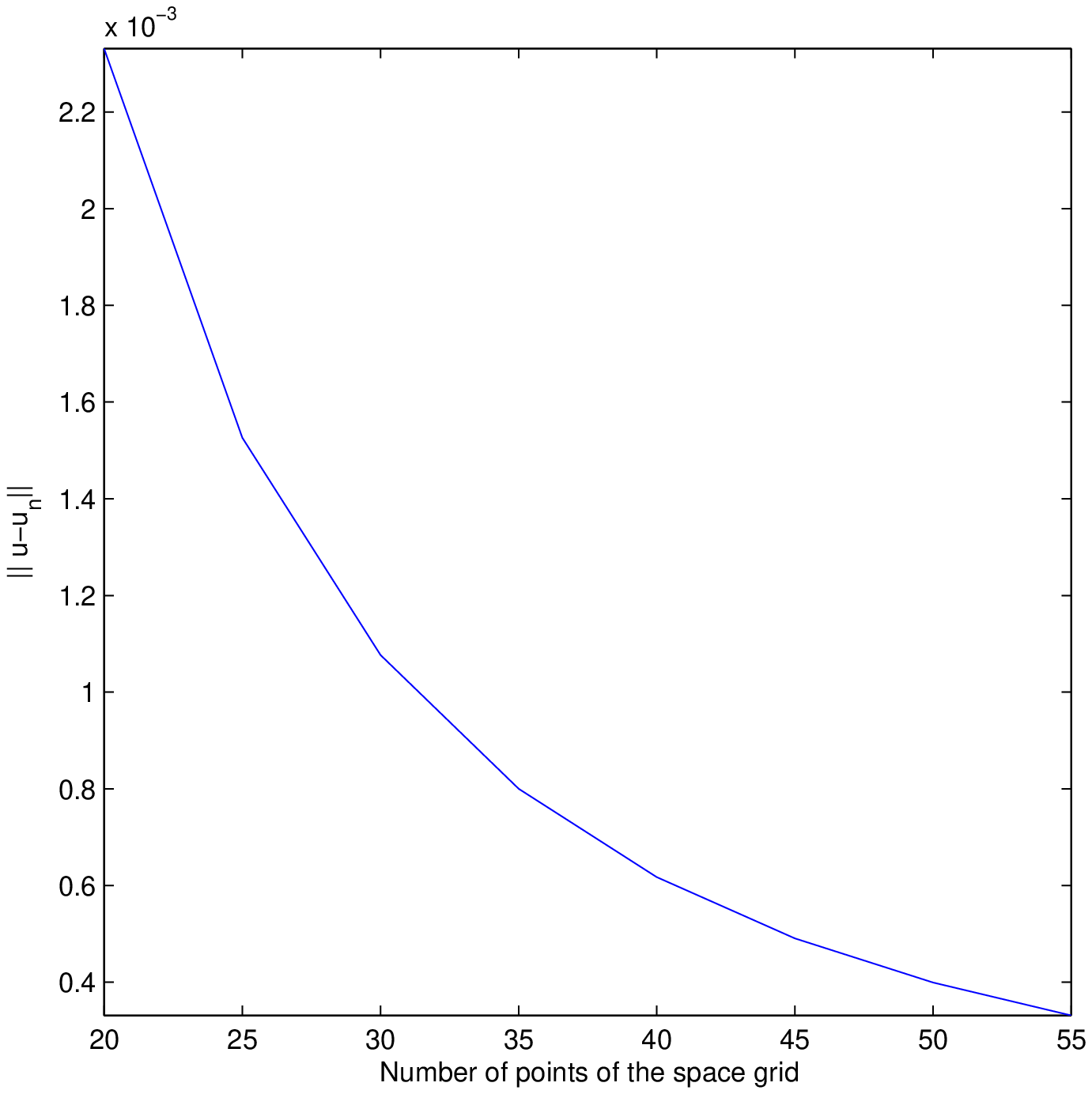}
    \caption{{\small $l_1-$error  for different space steps.}}
    \end{minipage}
    \begin{minipage}[htb]{0.5\linewidth}
    \includegraphics[width=\linewidth]{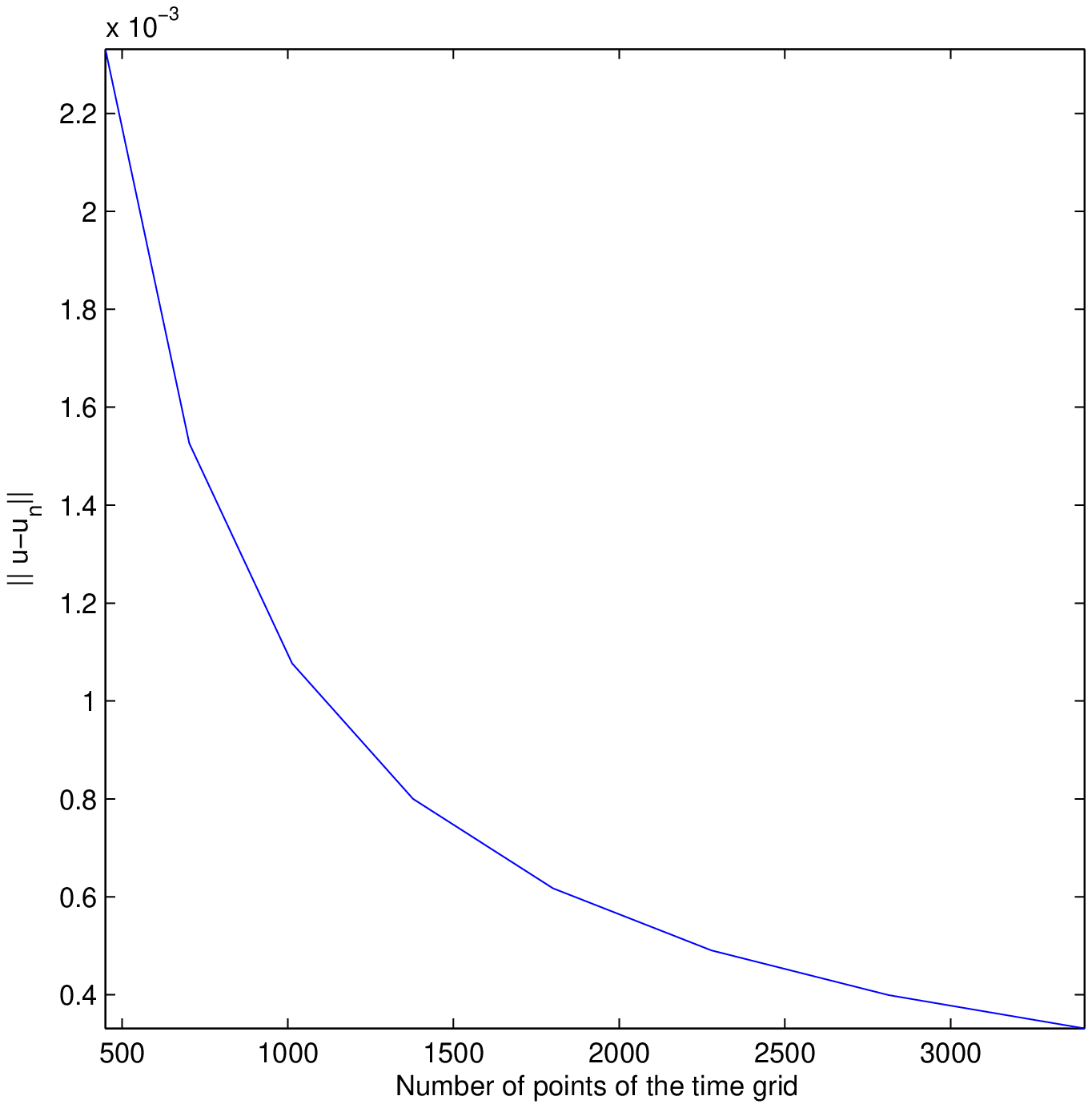}
    \caption{{\small $l_1-$error  for different time steps.}}
    \end{minipage}
    \end{figure}

\textbf{Example 2:}  In this example we consider an exact
complex-valued $C^\infty$ solution $u=(0,u_1,u_2,u_3)$ of
(\ref{nonlinearSchrodinger}), where
    \begin{gather*}
    u_1(x,t) = \left( e^{-t}-1 \right)~(x_1^2-25)~(x_2^2-25)~(x_3^2-25) \\
    u_2(x,t) = 0, ~~~u_3(x,t) = \left( e^{-t}-1 \right)~\sin(\pi
    x_1x_2x_3)~e^{ix_1t}.
    \end{gather*}

Below is the table with the error of approximation between the exact
solution $u$ and its discrete approximation $u_{h,\tau}$ on the
domain $\Omega=[-5,5]^3\times[0,2]$, for different mesh sizes,
\begin{center}\begin{small}
\begin{tabular}{|c|c|c|c|c|}
  \multicolumn{5}{c}{\textbf{Table 2}} \\
  \hline \hline
  N & M & t=0 & t=0.4 & t=0.8 \\
  \hline
  20 & 450   & 4.8846$\times 10^{-3}$ & 2.6819$\times 10^{-3}$ & 1.2024$\times 10^{-3}$  \\
  25 & 703   & 3.1692$\times 10^{-3}$ & 1.7323$\times 10^{-3}$ & 7.8152$\times 10^{-4}$  \\
  30 & 1013  & 2.2183$\times 10^{-3}$ & 1.2172$\times 10^{-3}$ & 5.4746$\times 10^{-4}$  \\
  35 & 1378  & 1.6404$\times 10^{-3}$ & 8.9923$\times 10^{-4}$ & 4.4166$\times 10^{-4}$  \\
  40 & 1800  & 1.2613$\times 10^{-3}$ & 6.9250$\times 10^{-4}$ & 3.1112$\times 10^{-4}$  \\
  45 & 2278  & 9.9988$\times 10^{-4}$ & 5.4847$\times 10^{-4}$ & 2.4668$\times 10^{-4}$  \\
  50 & 2813  & 8.1181$\times 10^{-4}$ & 4.4563$\times 10^{-4}$ & 2.0029$\times 10^{-4}$  \\
  55 & 3404  & 6.7227$\times 10^{-4}$ & 3.6914$\times 10^{-4}$ & 1.6580$\times 10^{-4}$  \\
  \hline \hline
  N & M & t=1.2 & t=1.6 & t=2 \\
  \hline
  20 & 450   & 5.3907$\times 10^{-4}$ & 2.4374$\times 10^{-4}$ & 1.1140$\times 10^{-5}$ \\
  25 & 703   & 3.5116$\times 10^{-4}$ & 1.5779$\times 10^{-4}$ & 7.1668$\times 10^{-5}$ \\
  30 & 1013  & 2.4623$\times 10^{-4}$ & 1.1033$\times 10^{-4}$ & 5.0000$\times 10^{-5}$ \\
  35 & 1378  & 9.0828$\times 10^{-4}$ & 8.1648$\times 10^{-5}$ & 3.6900$\times 10^{-5}$ \\
  40 & 1800  & 1.8166$\times 10^{-4}$ & 6.2800$\times 10^{-5}$ & 2.8334$\times 10^{-5}$ \\
  45 & 2278  & 1.1076$\times 10^{-4}$ & 4.9816$\times 10^{-5}$ & 2.4428$\times 10^{-5}$ \\
  50 & 2813  & 8.9900$\times 10^{-5}$ & 4.0406$\times 10^{-5}$ & 1.8210$\times 10^{-5}$ \\
  55 & 3404  & 7.4468$\times 10^{-5}$ & 3.3484$\times 10^{-5}$ & 1.5072$\times 10^{-5}$ \\
  \hline \hline
  \multicolumn{5}{c}{$l_1-$error between the approximated solution and the exact solution}\\
  \multicolumn{5}{c}{at different instants}
\end{tabular}
\end{small}\end{center}

\noindent followed by the graphics (Figures 3. and 4.) of the
evolution of the approximation error  for the correspondent space
and time mesh sizes  considered.
\begin{figure}[htb]
    \begin{minipage}[h]{0.5\linewidth}
    \includegraphics[width=\linewidth]{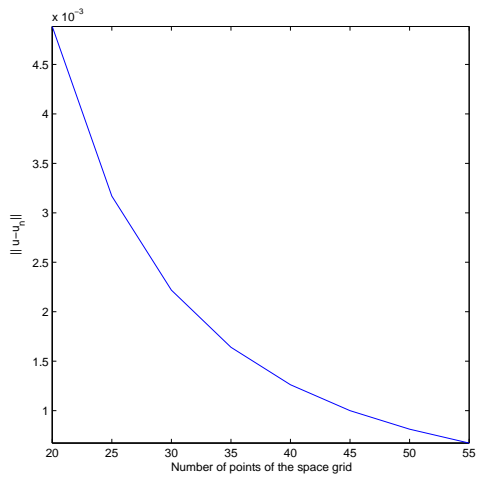}
    \caption{{\small $l_1-$error  for different space steps.}}
    \end{minipage}
    \begin{minipage}[htb]{0.5\linewidth}
    \includegraphics[width=\linewidth]{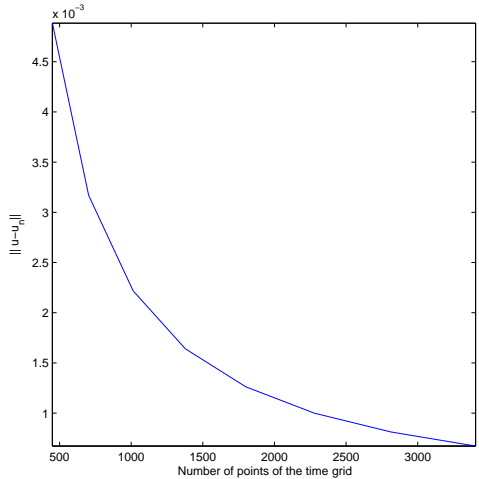}
    \caption{{\small $l_1-$error  for different time steps.}}
    \end{minipage}
    \end{figure}

\textbf{Example 3:} Finally, we conclude with an example of an exact
solution of lower regularity on the domain $\Omega=[-5,5]^3 \times
[0,2]$, namely an exact $C^1-$solution $u=(0,u_1,u_2,u_3)$ of
(\ref{nonlinearSchrodinger}), with
    \begin{gather*}
    u_1(x,t) = (e^{-t}-1)~(g(x_1)-g(-x_1))~(g(x_2)-g(-x_2))~(g(x_3)-g(-x_3)) \\
    u_2(x,t) = u_3(x,t) = 0,
    \end{gather*}
where $g$ is the auxiliary B-spline of order 3
    \begin{eqnarray*}
    g(y) & = & \left\{
    \begin{array}{ccl}
\frac{y^3}{6} & \mbox{if} & 0 \leq y < 1 \\ & & \\
-\frac{1}{3}+\frac{y}{2}+\frac{(y-1)^2}{2}-\frac{11(y-1)^3}{24} & \mbox{if} & 1 \leq y < 2 \\ & & \\
\frac{11}{24}+\frac{y}{8}-\frac{7(y-2)^2}{8}+\frac{3(y-2)^3}{8} & \mbox{if} & 2 \leq y < 3 \\ & & \\
\frac{11}{6}-\frac{y}{2}+\frac{(y-3)^2}{4}-\frac{(y-3)^3}{24} & \mbox{if} & 3 \leq y \leq 5 \\
    \end{array}
    \right..
    \end{eqnarray*}

Again, the corresponding right hand side $f = i \partial_t u -\Delta
u - |u^2|u.$ The following table gives the error of approximation
between the exact solution $u$ and its discrete approximation
$u_{h,\tau}$ for different mesh sizes considered.

\begin{center}\begin{small}
\begin{tabular}{|c|c|c|c|c|}
  \multicolumn{5}{c}{\textbf{Table 3}} \\
  \hline \hline
  N & M & t=0 & t=0.4 & t=0.8 \\
  \hline
  20 & 450   & 5.0846$\times 10^{-3}$ & 2.8819$\times 10^{-3}$ & 1.4024$\times 10^{-3}$  \\
  25 & 703   & 3.7149$\times 10^{-3}$ & 2.0388$\times 10^{-3}$ & 9.1607$\times 10^{-4}$  \\
  30 & 1013  & 2.7242$\times 10^{-3}$ & 1.4948$\times 10^{-3}$ & 6.7232$\times 10^{-4}$  \\
  35 & 1378  & 1.9355$\times 10^{-3}$ & 1.0610$\times 10^{-3}$ & 4.7688$\times 10^{-4}$  \\
  40 & 1800  & 1.4763$\times 10^{-3}$ & 8.1058$\times 10^{-4}$ & 3.6402$\times 10^{-4}$  \\
  45 & 2278  & 1.1856$\times 10^{-3}$ & 6.5030$\times 10^{-4}$ & 2.9248$\times 10^{-4}$  \\
  50 & 2813  & 9.1813$\times 10^{-4}$ & 4.5629$\times 10^{-4}$ & 2.0291$\times 10^{-4}$  \\
  55 & 3404  & 8.0086$\times 10^{-4}$ & 4.3975$\times 10^{-4}$ & 1.9751$\times 10^{-4}$  \\
  \hline \hline
  N & M & t=1.2 & t=1.6 & t=2 \\
  \hline
  20 & 450   & 7.3907$\times 10^{-4}$ & 2.4437$\times 10^{-4}$ & 1.9111$\times 10^{-4}$ \\
  25 & 703   & 4.1162$\times 10^{-4}$ & 1.8495$\times 10^{-4}$ & 8.4088$\times 10^{-5}$ \\
  30 & 1013  & 3.0239$\times 10^{-4}$ & 1.3550$\times 10^{-4}$ & 6.1402$\times 10^{-5}$ \\
  35 & 1378  & 2.1434$\times 10^{-4}$ & 9.6336$\times 10^{-5}$ & 4.3534$\times 10^{-5}$ \\
  40 & 1800  & 1.6362$\times 10^{-4}$ & 7.3510$\times 10^{-5}$ & 3.3166$\times 10^{-5}$ \\
  45 & 2278  & 1.3133$\times 10^{-4}$ & 5.9066$\times 10^{-5}$ & 2.6610$\times 10^{-5}$ \\
  50 & 2813  & 9.9006$\times 10^{-5}$ & 4.4078$\times 10^{-5}$ & 1.9410$\times 10^{-5}$ \\
  55 & 3404  & 8.8712$\times 10^{-5}$ & 3.9888$\times 10^{-5}$ & 1.7956$\times 10^{-5}$ \\
  \hline \hline
  \multicolumn{5}{c}{$l_1-$error between the approximated solution and the exact solution}\\
  \multicolumn{5}{c}{at different instants}
\end{tabular}
\end{small}\end{center}

The next graphics (Figures 5. and 6.) show the evolution of the
approximation error in $l_1-$norm for the different space mesh size
and time mesh size considered.
\begin{figure}[htb]
    \begin{minipage}[h]{0.5\linewidth}
    \includegraphics[width=\linewidth]{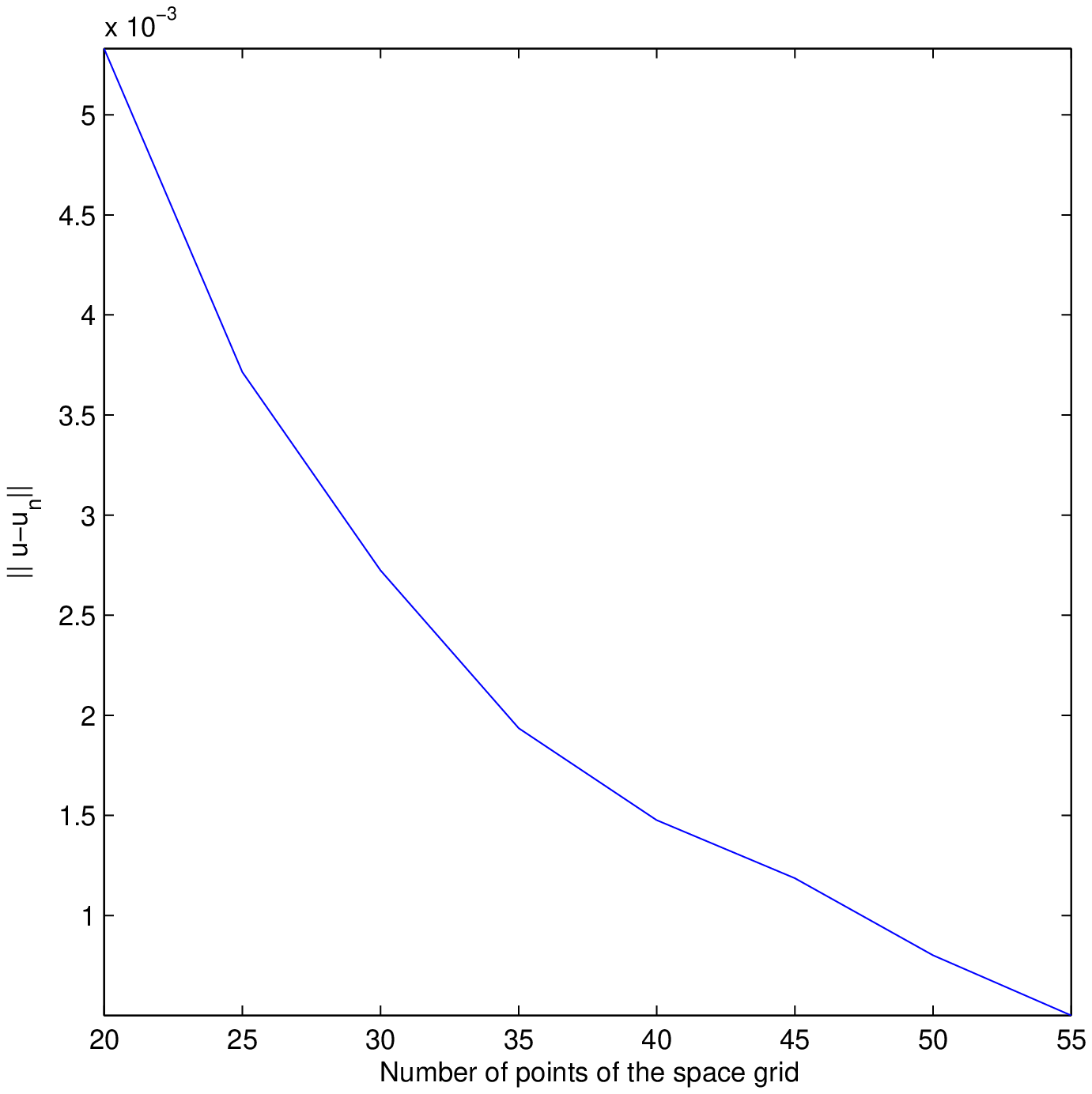}
    \caption{{\small $l_1-$error  for different space steps.}}
    \end{minipage}
    \begin{minipage}[htb]{0.5\linewidth}
    \includegraphics[width=\linewidth]{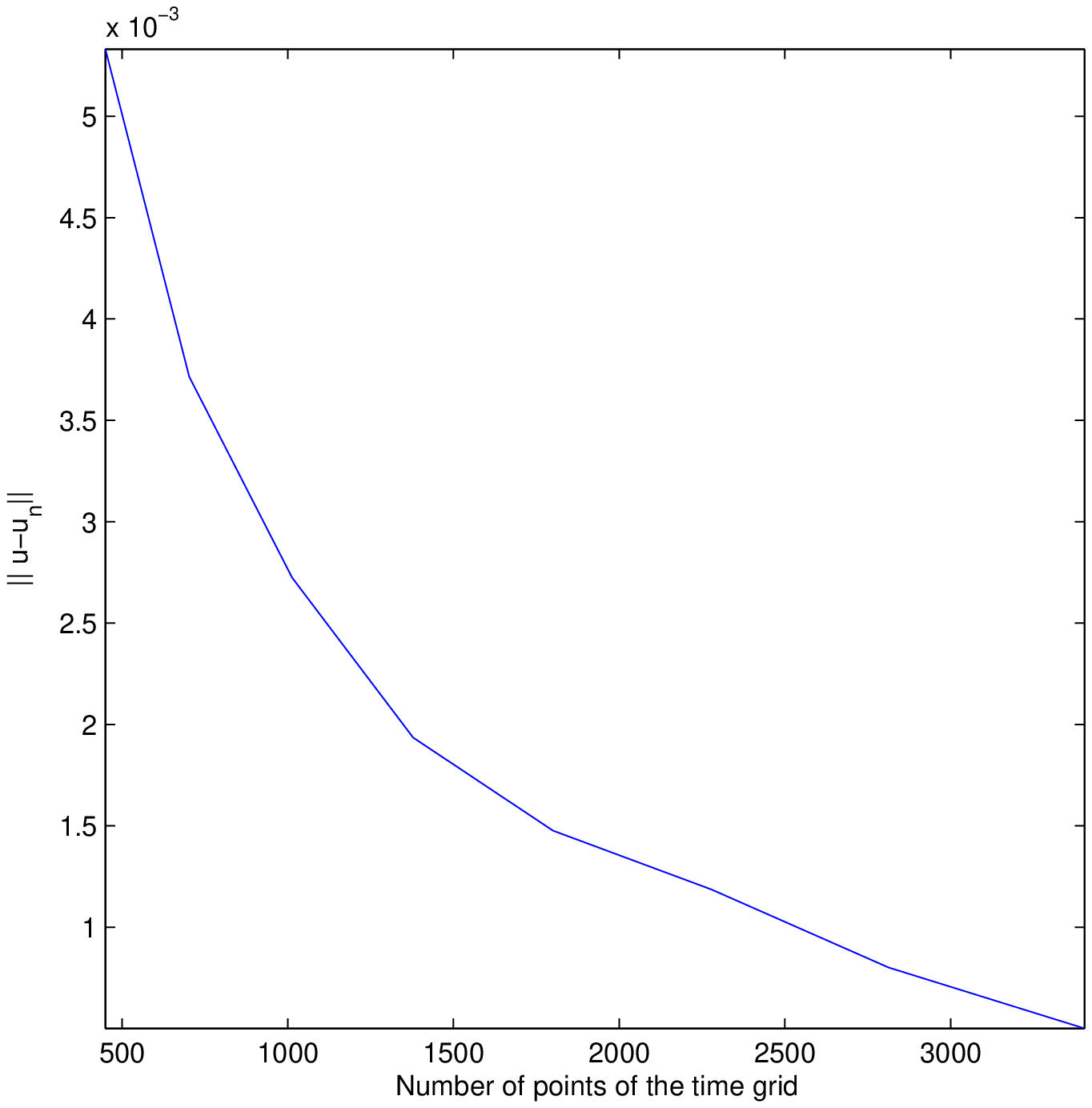}
    \caption{{\small $l_1-$error  for different time steps.}}
    \end{minipage}
    \end{figure}

Taking into account the previous graphics we are able to observe
that the order of convergence for the space coordinate is, in all
the examples, of order $\mathcal{O}(h^8)$, while for the time
coordinate we get, in all the examples, an order of convergence of
order $\mathcal{O}(\tau^{\frac{3}{2}})$. We remark that our method
seems to be stable under functions of lower regularity, since the
order of convergence for the space and time coordinates remains same
in all the three examples.


~\\

 {\bf Acknowledgement} The research of the first author was
(partially) supported by {\it Unidade de Investiga\c c\~ao
``Matem\'atica e Aplica\c c\~oes''} of the University of Aveiro. The
work of the second and third authors was supported by PhD-grants
\texttt{SFRH/BD/17657/2004}, \texttt{SFRH/BD/22646/2005}, of {\it
Funda\c c\~ao para a Ci\^encia e a Tecnologia}.


\end{document}